\documentclass{article}

\usepackage{amsmath}
\usepackage{amsthm}
\usepackage{amsfonts}
\usepackage{amsbsy}
\usepackage{amssymb}
\usepackage[initials]{amsrefs}
\newtheorem{theorem}{Theorem}

\newtheorem{lemma}{Lemma}

\newcommand{\R}{{\mathbb R}}

\newcommand{\set}[2]{ \left\{ #1 \ \left| \ #2 \right. \right\} }

\usepackage{multirow}

\title{$L^p$-nondegenerate Radon-like operators with vanishing rotational curvature}
\author{Philip T. Gressman\footnote{Partially supported by NSF grant DMS-1101393 and an Alfred P. Sloan Foundation Fellowship.}}
\begin{document}
\maketitle

\begin{abstract}
We consider the $L^p \rightarrow L^q$ mapping properties of a model family of Radon-like operators integrating functions over $n$-dimensional submanifolds of $\R^{2n}$.  %Any such operators with nonvanishing rotational curvature must map $L^{\frac{3}{2}}$ to $L^3$.  
It is shown that nonvanishing rotational curvature is never generic when $n \geq 2$ and is, in fact, impossible for all but finitely many values of $n$.  Nevertheless, operators satisfying the same $L^p \rightarrow L^q$ estimates as the ``nondegenerate'' case (modulo the endpoint) are dense in the model family for all $n$. 
\end{abstract}

\section{Main Results}

In this paper, we present a curious fact about Radon-like operators.  When studying mapping properties from $L^2$ into the scale of $L^2$-Sobolev spaces, it has long been known that the correct way to identify ``nondegenerate'' operators is via the nonvanishing of the Phong-Stein rotational curvature \cite{ps1986I, ps1986II} (we will remind the reader of the details shortly). In this paper, we show that this correspondence does not hold true when one instead considers optimal $L^p$-improving estimates. 

The specific class of operators we will consider are averages over certain families of $n$-dimensional submanifolds of $\R^{2n}$. Fix a dimension $n$, and let $Q : \R^n \times \R^n \rightarrow \R^n$ be bilinear; for convenience we define $Q_{ijk}$ so that %We may identify all such choices $Q$ with points in $\R^{n^3}$ by recording the coefficients of the quadratic form in the standard basis:
\[ Q(x,y) = \left( \sum_{j,k=1}^n Q_{1jk} x_j y_k, \ldots, \sum_{j,k=1}^n Q_{njk} x_j y_k \right) \]
in the standard basis.
To this $Q$ we may associate an averaging operator acting on functions of $\R^n \times \R^n$ via the formula
\begin{equation} T_Q f (x,x') := \int_{[-1,1]^n} f(x+t, x' + Q(x,t)) dt. \label{mainclass} \end{equation}
Though this operator is only semi-translation-invariant in general, when $Q$ is symmetric, conjugating by  the operator $R_Q$, defined by $R_Q f(x,x') := f(x, x' - \frac{1}{2} Q(x,x))$ (which preserves all $L^p$-norms), gives
 \begin{equation} R_Q^{-1} T_Q R_Q f(x,x') = \int_{[-1,1]^n} f\left(x+t, x' - \frac{1}{2} Q(t,t) \right) dt, \label{convo} \end{equation}
 which is convolution with the standard measure on the quadratic submanifold parametrized by $(t, \frac{1}{2} Q(t,t))$.  Thus, when $Q$ is symmetric, the operator $T_Q$ will satisfy the same $L^p \rightarrow L^q$ estimates as the convolution operator \eqref{convo}.  Similarly, we may effectively consider this class of operators \eqref{mainclass} to be closed under adjoints: for any choice of $Q$, the adjoint operator $T_Q^*$ will satisfy the same $L^p \rightarrow L^q$ estimates as the operator $T_{Q^*}$, where $Q^*(x,t) := Q(t,x)$.
 
The purpose of this paper is to demonstrate the failure of rotational curvature to characterize the ``$L^p$-nondegeneracy'' of these operators.  To that end, there are two goals.  The first is to establish that essentially all operators of the form \eqref{mainclass} satisfy the best-possible range of $L^p \rightarrow L^q$ estimates (up to possibly the $L^{\frac{3}{2}} \rightarrow L^3$ endpoint itself). The second is to show that nonvanishing Phong-Stein rotational curvature is sufficient but \eqref{mainclass} not generally necessary to establish optimal $L^{p} \rightarrow L^q$ estimates.  We summarize these results in the main theorem:
\begin{theorem} 
Fix any bilinear $Q$ as described above.  For each $u \in \R^n$, define
\[ Q_u (x,y) := Q(x,y) + \sum_{j=1}^n u_i x_i y_i. \]
Then for almost every $u \in \R^n$ (and consequently, for a dense set of $u \in \R^n$), the operator $T_{Q_u}$ maps $L^p(\R^{2n})$ to $L^q(\R^{2n})$
whenever $(\frac{1}{p}, \frac{1}{q})$ belongs to the closed triangle with vertices $(0,0), (1,1)$, and $(\frac{2}{3}, \frac{1}{3})$, with the possible exception of the %case when $(\frac{1}{p}, \frac{1}{q})$ exactly equals 
the vertex $(\frac{2}{3}, \frac{1}{3})$. The set of $u$ for which $T_{Q_u}$ exhibits nonvanishing rotational curvature will be open in $\R^n$ but never dense when $n \geq 2$, and will in fact be empty unless $n=1,2,4,$ or $8$.  \label{mainthm}
\end{theorem}
Typical Knapp-type examples show that no larger range of exponents can hold. A corollary corollary of the main theorem is that when each $Q$ is identified with a point in $\R^{n^3}$ by its coordinates $Q_{ijk}$, the optimal range of estimates (modulo endpoint) holds for a dense subset of $\R^{n^3}$.   The same statement is true for convolutions with measures on $n$-dimensional quadratic submanifolds in $\R^{2n}$, for the simple reason that the perturbed $Q_u$ are all symmetric and so also correspond to such convolutions.  It is not known whether the operators $Q_u$ to which theorem \ref{mainthm} applies form an open subset of $\R^n$ for every fixed $Q$.

The reader may be interested in a ``pointwise'' criterion by which one may guarantee that a particular $T_Q$ satisfies the range of estimates identified in theorem \ref{mainthm}.  The particular criterion established in this paper is the following:

\begin{theorem}
Suppose $\Phi : \R^n \times \R^n \rightarrow \R^n$ is smooth and $\Phi(x,t)$ is linear in $t$.  For each $x$, let $J_{\Phi}(x,t)$ be the absolute value of the determinant of the map $u \mapsto (u \cdot \nabla_x) \Phi(x,t)$.
Suppose also that for each fixed $t$, the equation $\Phi(x,t) = c$ has only boundedly many solutions with $J_\Phi(x,t) \neq 0$ (note that this is immediately true when $\Phi$ is bilinear).
  If for each $\theta \in [0, 1)$, there is a constant $C_\theta$ such that \label{pointwise}
\begin{equation} \set{ t \in [-1,1]^n}{ J_\Phi(x,t) \leq \epsilon} \leq C_\theta \epsilon^\theta \label{sublevel} \end{equation}
uniformly for all $x$ and all $\epsilon > 0$, then the operator
\begin{equation} T f(x,x') := \int_{[-1,1]^n} f(x+t, x' + \Phi(x,t)) dt \label{corop} \end{equation}
maps $L^p \rightarrow L^q$ for all pairs $(\frac{1}{p}, \frac{1}{q})$ on the line segment between $(\frac{2}{3}, \frac{1}{3})$ and $(0,0)$ except possibly the endpoint $(\frac{2}{3},  \frac{1}{3})$.  
\end{theorem}

The proofs of theorems \ref{mainthm} and \ref{pointwise}, though relatively short, are postponed until the next section.
The  reader will no doubt be aware that, when the rotational curvature of the operator \eqref{corop} is nonvanishing, the conclusion of  theorem \ref{pointwise} (including the endpoint case) is obtainable by a relatively standard interpolation argument (see, for example, pgs. 427--428 of Stein \cite{steinha}).  The novelty of theorem 2, then, is that the hypothesis \eqref{sublevel} is strictly weaker than the assumption of nonvanishing rotational curvature.  The concept of rotational curvature was introduced by Phong and Stein \cite{ps1986I,ps1986II} based on earlier ideas from the study of FIOs \cite{gs1977}. The nonvanishing of rotational curvature is well-known to characterize the nondegeneracy of Radon-like averaging operators over submanifolds of codimension one for both optimal Sobolev regularity and $L^p$-improvement; for interesting related characterizations, see Oberlin \cite{oberlin2000II}.   Rotational curvature can be defined for averaging operators of any codimension (the reader should see Greenleaf and Seeger \cite{gs2002} for a very nice discussion of the rotational curvature condition and its relationship to FIO theory), but as is frequently noted in the literature, it must vanish identically when codimension exceeds the dimension of the submanifold.  Thus the case that we consider here is critical for rotational curvature.

As soon as the codimension exceeds one, there are two important issues that arise.  The first is that the nonvanishing of rotational curvature no longer generically holds.  Specifically, if one considers product-type averaging operators
\[ T f (x,x') := \! \int_{[-1,1]^n} \! \! f(x + t, x_1' + t_1^2,\ldots, x_n'+t_n^2) ~ dt_1 \cdots dt_n, \]
not only does the rotational curvature vanish at every point, but it is relatively easy to show that any sufficiently small perturbation of the underlying family of submanifolds will continue have vanishing rotational curvature at every point.  Modifications of this example establish the same failure of rotational curvature to be generic whenever the codimension is at least two.  

More distressing than the failure of nonvanishing rotational curvature to be generic is that typically there will be no family of submanifolds whatsoever with nonvanishing rotational curvature unless the dimension $d$ and the codimension $n'$ satisfy extremely strict number-theoretic compatibility constraints.  Related examples of such constraints are well-known---in his thesis, Christ \cite{christthesis} observes that, in the case of submanifolds of codimension 2, optimal Fourier restriction estimates change when the ambient dimension is odd versus even. More recently, Bourgain and Guth \cite{bg2011} established a restriction theorem which depends on the dimension mod 3.  It does not appear, however, that a complete listing of the possibility or impossibility of nonvanishing rotational curvature in terms of dimension and codimension has ever been undertaken in the harmonic analysis literature.  In reality, it is much more common (in terms of dimension and codimension) for Radon-like operators to be degenerate (in the Phong-Stein sense) than nondegenerate.  
Specifically, fix open sets $\Omega_L, \Omega_R \subset \R^d$ and let $\mathcal I$ be an incidence relation on $\Omega_L \times \Omega_R$ given by some smooth defining function $\Phi : \Omega_L \times \Omega_R \rightarrow \R^{n'}$, i.e.,
 \[ {\mathcal I} := \set{ (x,y) \in \Omega_L \times \Omega_R}{ \Phi(x,y) = 0 }. \]  We suppose that the natural projections $\pi_L : {\mathcal I} \rightarrow \Omega_L$ and $\pi_R : {\mathcal I} \rightarrow \Omega_R$ have surjective differentials.  We can define a Radon-like operator via
 \begin{equation} R f(x) := \int_{ \set{y \in \Omega_R}{\Phi(x,y) = 0}} f(y) d \sigma(y) \label{ps} \end{equation}
 where $d \sigma(y)$ is the induced Lebesgue measure.  Unless $d$ and $n'$ satisfy a strict number-theoretic compatibility relation, the operator \eqref{ps} cannot generally be nondegenerate in the Phong-Stein sense:
 
 \begin{theorem}
Suppose that $d-n'$ (i.e., the dimension of the submanifolds) factors into the form $2^{4q + r} s$ for integers $q,r,s,$ such that $s$ is odd and $0 \leq r \leq 3$.  Then the operator  \eqref{ps} must have vanishing rotational curvature when $n'$ (i.e, the codimension of the submanifolds) exceeds $8q + 2^r$.  If the \label{numthm} codimension does not exceed $8q + 2^r$, then there exist operators of the form \eqref{ps} with nonvanishing rotational curvature. %$n' \leq 8 q + 2^{r}$.  
\end{theorem}
\begin{proof}
Make some choice of coordinate systems on $\Omega_L$ and $\Omega_R$. The Phong-Stein rotational curvature condition requires the invertibility of the Monge-Ampere type matrix
\[ \left( \begin{array}{cccccc} \sum_{i=1}^{n'} \lambda_i \frac{\partial^2 \Phi_i}{\partial x_1 \partial y_1} & \cdots & \sum_{i=1}^{n'} \lambda_i \frac{\partial^2 \Phi_i}{\partial x_1 \partial y_d} & \frac{\partial \Phi_1}{\partial x_1} & \cdots & \frac{\partial \Phi_{n'}}{\partial x_1}  \\
 \vdots & \ddots & \vdots & \vdots & \ddots & \vdots \\
\sum_{i=1}^{n'} \lambda_i \frac{\partial^2 \Phi_i}{\partial x_d \partial y_1} & \cdots & \sum_{i=1}^{n'} \lambda_i \frac{\partial^2 \Phi_i}{\partial x_d \partial y_d} & \frac{\partial \Phi_1}{\partial x_d} & \cdots & \frac{\partial \Phi_{n'}}{\partial x_d}  \\
\frac{\partial \Phi_1}{\partial y_1} & \cdots & \frac{\partial \Phi_1}{\partial y_d} & 0 & \cdots & 0 \\
\vdots & \ddots & \vdots & \vdots & \ddots & \vdots  \\
\frac{\partial \Phi_{n'}}{\partial y_1} & \cdots & \frac{\partial \Phi_{n'}}{\partial y_d} & 0 & \cdots & 0
\end{array} \right) \]
for any choice $(\lambda_1,\ldots,\lambda_{n'}) \in \R^{n'} \setminus \{0\}$.  If we assume that the coordinate systems are chosen so that $\frac{\partial \Phi_i}{\partial x_j} = \frac{\partial \Phi_i}{\partial y_j} =  0$ for $1 \leq i \leq n'$ and $1 \leq j \leq d-n'$, then immediately the surjectivity of the differentials $d \pi_L$ and $d \pi_R$ makes this question equivalent to the invertibility of the matrix
\begin{equation} \left( \begin{array}{ccc} \sum_{i=1}^{n'} \lambda_i \frac{\partial^2 \Phi_i}{\partial x_1 \partial y_1} & \cdots & \sum_{i=1}^{n'} \lambda_i \frac{\partial^2 \Phi_i}{\partial x_1 \partial y_{d-n'}}  \\
 \vdots & \ddots & \vdots  \\
\sum_{i=1}^{n'} \lambda_i \frac{\partial^2 \Phi_i}{\partial x_{d-n'} \partial y_1} & \cdots & \sum_{i=1}^{n'} \lambda_i \frac{\partial^2 \Phi_i}{\partial x_{d-n'} \partial y_{d-n'}}
 \end{array} \right) \label{matrix} \end{equation}
 for all nonzero $\lambda \in \R^{n'}$.  %In general it is extremely difficult to find large families of matrices to play the role of the mixed Hessians of $\Phi_1,\ldots,\Phi_{n'}$: 
 In other words, we must find $n'$ real matrices of size $(d-n') \times (d-n')$ such that any nontrivial linear combination of these matrices is invertible.  
 If and when any such family of matrices exists, it is straightforward to construct an appropriate $\Phi$, linear in both $x$ and $y$, which will have nonvanishing rotational curvature.
 
 It does not appear to be widely-known in the harmonic analysis literature, but the existence of such families of matrices has been completely characterized for some time.  We refer the reader to Theorem A of Friedland, Robbin, and Sylvester \cite{frs1984}, obtained along the way during their complete characterization of strictly hyperbolic systems of first-order PDEs. The theorem establishes the equivalence of a variety of related algebraic and topological statements; one of which (condition (A4)) is the existence of a family of matrices with the invertibility properties we seek, and another (condition (A1)) is the number-theoretic constraint in the statement of the present theorem: $n' \leq 8 q + 2^r$.  The methods found in \cite{frs1984} are quite different from those to be taken up shortly, so we refrain from reproducing that proof here.
 \end{proof}

Let us pause for a moment to observe the severity of the constraint identified in theorem \ref{numthm}.  For convenience, define $n := d-n'$. It is frequently observed in the literature that the rotational curvature condition cannot be satisfied when $n' > n$.  Less frequently observed is that codimension $n'$ must equal one when $n$ is odd (which is essentially the phenomenon observed by Christ \cite{christthesis}), simply because the determinant of \eqref{matrix} acquires a minus sign when $\lambda$ is replaced by $- \lambda$ (so that the determinant must vanish at some point on the unit sphere in $\R^{n'}$).  In general, the codimension $n'$ can only grow when the dimension $n$ is divisible by a large power of $2$, and even in this case the codimension is constrained to grow logarithmically in the dimension.  The special case $n = n'$ is allowed only when $n=1,2,4$, or $8$. An interesting, if somewhat unexpected, alternate proof of this fact involves using a nondegenerate collection of $n$ matrices of size $n \times n$ to construct the multiplication table for a division algebra on $\R^n$.  Along these same lines, positive examples of matrices satisfying these properties can be constructed by considering the action of $n$ generators on the Clifford algebra $Cl_n$ associated to any definite quadratic form.

%
%
%Let $Q$ be any sequence of $n$ real symmetric $n \times n$ matrices $Q_1,\ldots,Q_n$.
%\[ \Phi_Q(t) := \left( \begin{array}{c} \sum_{1 \leq i, j \leq n} Q_{1ij} t_i t_j \\ \vdots \\  \sum_{1 \leq i, j \leq n} Q_{n ij} t_i t_j  \end{array} \right). \]
%
%\[ XY := \left( \begin{array}{c} \sum_{1 \leq i , j \leq n} Q_{ij1} X_i Y_j \\ \vdots \\ \sum_{1 \leq i, j \leq n} Q_{ijn} X_i Y_j \end{array} \right). \]
%This product clearly distributes on the left and the right and satisfies the property that $XY = 0$ if and only if $X = 0$ or $Y = 0$.  Consequently this product generates a real division algebra, and so the dimension $n = 1, 2, 4$, or $8$.
%
%
%
%
%
%\newpage

\section{Proofs of theorems \ref{mainthm} and \ref{pointwise}}

%\cite{stovall2011} \cite{tw2003} \cite{gressman2013} Dendrinos, Laghi, and Wright \cite{dlw2009}

At its core, the proof of theorems \ref{mainthm} and \ref{pointwise} is a variation on the method of inflation due to Christ \cite{christ1998} (see \cite{tw2003,eo2010,stovall2011,dlw2009,gressman2013} for a variety of recent results which extend Christ's ideas in various directions). Because our class of averaging operators is effectively closed under adjoints (and clearly bounded on $L^\infty$), it suffices to establish restricted $L^p \rightarrow L^q$ estimates for a sequence of pairs $(\frac{1}{p_m}, \frac{1}{q_m})$ tending to $(\frac{2}{3},\frac{1}{3})$ along the line segment from $(0,0)$ to $(\frac{2}{3}, \frac{1}{3})$.

Specifically, suppose that $n$, $n'$, and $k$ are integers such that $n = n'(k-1)$. Let $x' \in \R^{n'}$, $x, t^{(1)},\ldots,t^{(k)} \in \R^n$, and let $\Phi$ be a smooth map from $\R^n \times \R^n$ to $\R^{n'}$ (note the argument that follows may easily be modified if $\Phi$ is defined only on an open subset of $\R^n \times \R^n$).  We consider the map
\begin{equation} 
\begin{split}
(x,x', t^{(1)},\ldots,& t^{(k)}) \mapsto \\ & \left( x + t^{(1)}, x' + \Phi(x, t^{(1)}), \ldots, x + t^{(k)}, x' + \Phi(x, t^{(k)}) \right).
\end{split} \label{mapping} \end{equation}
We may write the Jacobian matrix of this mapping in block form (with the various coordinate components of the mapping \eqref{mapping} corresponding to rows and derivatives with respect to variables $(x,x',t^{(1)},\ldots,t^{(k)})$ corresponding to columns) as
\[ \left[ \begin{array}{cc|ccccc}  
I_{n \times n} & 0 & I_{n \times n}   & 0 &  \cdots & \cdots & 0 \\ 
\frac{\partial \Phi}{\partial x}(x,t^{(1)}) & I_{n' \times n'} & \frac{\partial \Phi}{\partial t}(x, t^{(1)}) & \ddots & \ddots & \ddots & \vdots \\  
\vdots  & 0 & 0 & \ddots & \ddots & \ddots & \vdots \\ 
\vdots  & \vdots & \vdots & \ddots & \ddots & \ddots & 0 \\ 
 I_{n \times n} &  0 &  \vdots &  \ddots & \ddots & \ddots & I_{n \times n}    \\ 
\frac{\partial \Phi}{\partial x}(x,t^{(k)}) & I_{n' \times n'}  & 0 & \cdots & \cdots &0  &  \frac{\partial \Phi}{\partial t}(x, t^{(k)})  \\  

%\multirow{2}{*}{$I_{d \times d}$} & \vdots & \ddots  &  \ddots & I_{n \times n}  \\  & 0 & \cdots  &   0 & \frac{\partial \Phi}{\partial t}(t^{(k)}) 
%\end{array} \right]   \]
%where by $\frac{\partial \Phi}{\partial t}$ we mean the $n' \times n$ Jacobian matrix of $\Phi$.
%Subtracting the various $t$-derivative columns from the corresponding $x$-derivative columns gives
%\[ \left[ \begin{array}{ccccccc}  
%0 & 0   & I_{n \times n} & 0 &  \cdots & 0 \\  - \frac{\partial \Phi}{\partial t}(t^{(1)})  & I_{n' \times n'} & \frac{\partial \Phi}{\partial t}(t^{(1)}) & \ddots & \ddots & \vdots \\  
%\vdots & \vdots & 0 & \ddots & \ddots & 0 \\ 
%0 & 0 & \vdots & \ddots  &  \ddots   & I_{n \times n}  \\ -\frac{\partial \Phi}{\partial t}(t^{(k)})  & I_{n' \times n'} & 0 & \cdots  &   0 & \frac{\partial \Phi}{\partial t}(t^{(k)}) 
\end{array} \right].   \]
To compute the determinant, we take the columns on the right of the vertical divider (corresponding to derivatives with respect to the $t$'s) and subtract from the initial columns (derivatives with respect to $x$) to eliminate all the $I_{n \times n}$ squares to the left of the divider. Then we expand in the rows containing the remaining $I_{n \times n}$ blocks (i.e., the rows corresponding to components of $x + t^{(i)}$).  We conclude that the absolute value of the Jacobian determinant of \eqref{mapping} equals
\begin{equation} J_{\Phi}(x, t^{(1)},\ldots,t^{(k)}) := \left| \det \left[ \begin{array}{cc}  \left( \frac{\partial}{\partial x} - \frac{\partial}{\partial t} \right) \Phi (x, t^{(1)})  & I_{n' \times n'} \\ \vdots & \vdots \\ \left( \frac{\partial}{\partial x} - \frac{\partial}{\partial t} \right) \Phi(x, t^{(k)})  & I_{n' \times n'} \end{array} \right] \right|. \label{jacobian}
\end{equation}

With the Jacobian determinant now identified, we prove a lemma which relates certain sublevel sets constructed from the Jacobian to a corresponding Radon-like operator:
\begin{lemma}
Fix $\Omega \subset \R^n$ to have finite measure. \label{estimate}
Suppose that on $\R^n \times \R^{n'} \times \Omega^{k}$, the mapping \eqref{mapping} has nondegenerate multiplicity at most $N$ (i.e., among any $N+1$ distinct points mapping to the same value under \eqref{mapping}, at least one of the points belongs to the zero set of the Jacobian \eqref{jacobian}, where we interpret the $J_\Phi$ to be constant as a function of $x'$).
Suppose 
\begin{align*}
 \left| \set{t \in \R^n} {J_{\Phi}(x,t, t^{(1)}, \ldots, \widehat{t^{(i)}}, \ldots,t^{(k)}) \leq J_{\Phi}(x,t^{(1)},\ldots,t^{(k)}) \ \forall i=1,\ldots,k \! } \right| \\ \leq C J_{\Phi}^{\theta} (x,t^{(1)} ,\ldots,  t^{(k)}&)
 \end{align*}
uniformly for all choices of $x,t^{(1)},\ldots,t^{(k)}$, for some fixed $\theta \in (0,1]$ (here circumflex denotes omission).  Then the convolution operator $T$ given by
\[ T f(x) := \int_{\Omega} f(x +t, x' + \Phi(x,t)) dt \]
is of restricted type $( \frac{k+1}{k} \theta^{-1}, (k+1) \theta^{-1})$, i.e.,
\[ || T \chi_F||_{\frac{k+1}{\theta}} \leq C^{\frac{1}{k+1}} (k+1)^{\frac{1}{k+1}} N^{\frac{\theta}{k+1}} |\Omega|^{\frac{k(1-\theta)}{k+1}} |F|^\frac{k \theta}{k+1}. \]
\end{lemma}

\begin{proof}
We fix an arbitrary measurable set $F$ and expand:
\begin{align*}
 &( T \chi_F(x,x'))^{k+1}  \\   & = \int_{\Omega^{k+1}} \!  \prod_{j=0}^k \left[ \chi_{F}(x + t^{(j)}, x'+ \Phi(x,t^{(j)})) \right]  dt^{(0)} \cdots dt^{(k)} \\
 & \leq (k+1)\int_{\Omega^{k+1} \cap S^{(0)}} \!  \!\prod_{j=0}^k \!  \left[ \chi_{F}(x + t^{(j)}, x' + \Phi(x,t^{(j)})) \right] \!  dt^{(0)} \cdots dt^{(k)}, 
 \end{align*}
 where, by symmetry, $S^{(0)}$ is the subset of $\Omega^{k+1}$ on which 
 \[ J_{\Phi}(x, t^{(0)}, \ldots, \widehat{t^{(i)}}, \ldots,t^{(k)}) \leq J_{\Phi}(x,t^{(1)},\ldots,t^{(k)}) \]
 for each $i=1,\ldots, k$.  By hypothesis, we know that the set of points $t^{(0)}$ for which these inequalities are true has measure at most $C J_\Phi^{\theta}(t^{(1)},\ldots,t^{(k)})$ when these remaining variables are considered fixed.  If we denote $J_{\Phi}(x, t^{(1)},\ldots,t^{(k)})$ simply by $J(x,t)$, then we have
 \begin{equation}
 \begin{split}
   & ( T \chi_F(x))^{k+1} \leq C'   \int_{\Omega^{k}} \! \prod_{j=1}^k \left[ \chi_{F}(x + t^{(j)}, x'+\Phi(t^{(j)})) \right] J^{\theta}(x,t) dt^{(1)} \cdots dt^{(k)}
\end{split} \label{midarg}
\end{equation}
with $C' := (k+1) C$.
Now apply H\"{o}lder's inequality to this integral, raising the term $J^\theta$ to the $\frac{1}{\theta}$ power and the constant function $1$ to the dual power, to conclude that $T \chi_F(x)$ is bounded above by
 \begin{align*}
 {C'}^\frac{1}{k+1}   |\Omega|^\frac{k (1-\theta)}{k+1} \left[ \int_{\Omega^{k}} \! \prod_{j=1}^k \left[ \chi_{F}(x + t^{(j)}, x'+\Phi(t^{(j)})) \right] J(x,t) dt^{(1)} \cdots dt^{(k)} \right]^\frac{\theta}{k+1}.
\end{align*}
Finally, observe that the quantity inside the brackets has integral at most $N |F|^k$ when integrated with respect to $x$ and $x'$ by the change of variables formula.  Thus we conclude that
\[ ||T \chi_F||_{\frac{k+1}{\theta}} \leq {C'}^\frac{1}{k+1}   |\Omega|^\frac{k (1-\theta)}{k+1} \left( N |F|^{k} \right)^{\frac{\theta}{k+1}} \]
exactly as desired.
%
%
%If we now apply H\"{o}lder's inequality and use the fact that the nondegenerate multiplicity of \eqref{mapping} is at most $N$, we can bound the right-hand side of \eqref{midarg} above by
%\begin{align*}
%C'  ||g||_{\frac{1}{1-\theta}} |\Omega|^{k(1-\theta)} \! & \left[ \int_{\R^d}  \! \int_{\Omega^{k}} \! \prod_{j=1}^k \left[ \chi_{F}(x + (t^{(j)}, \Phi(t^{(j)}))) \right] J (t) dt^{(1)} \cdots dt^{(k)} dx \right]^{\theta} \\
%\leq C' & ||g||_{\frac{1}{1-\theta}} |\Omega|^{k(1-\theta)} N^{\theta} |F|^{k \theta}.
%\end{align*}
%Finally, let $\rho := (1-\theta)/(1-\theta + k)$. 
%\begin{align*}
% \int_{\R^d} |g(x)| T \chi_F(x) dx   & \leq \int_{\R^d} \left[ |g(x)|^\rho T \chi_{F}(x) \right] \left[ |g(x)|^{1- \rho} \right] dx \\
%& \leq \left[ \int |g(x)|^{\rho (k+1)} (T \chi_F(x))^{k+1} dx \right]^{\frac{1}{k+1}} || |g|^{1-\rho}||_{\frac{k+1}{k}} \\
% & \leq \left[ C' |||g|^{\rho(k+1)}||_{\frac{1}{1-\theta}} |\Omega|^{k(1-\theta)} N^{\theta} |F|^{k \theta} \right]^{\frac{1}{k+1}} || g||_{\frac{(k+1)(1-\rho)}{k}}^{1-\rho} \\
% & \leq C'^{\frac{1}{k+1}} N^{\frac{\theta}{k+1}} |\Omega|^{\frac{k(1-\theta)}{k+1}} |F|^{\frac{k \theta}{k+1}} ||g||_{\frac{k+1}{k+1-\theta}}. 
%\end{align*}
%Since $g$ is arbitrary, by duality the lemma follows.
%%\[ || T f||_{\frac{k+1}{\theta}} \leq C^{\frac{1}{k+1}} (k+1)^{\frac{1}{k+1}} N^{\frac{\theta}{k+1}} |\Omega|^{\frac{k(1-\theta)}{k+1}}||f||_{\frac{k+1}{k \theta},1}. \]
\end{proof}

The passage from lemma \ref{estimate} to theorem \ref{pointwise} is fairly quick.  We take $d=2n$, $n' = n$, and $k=2$.  All that one needs to verify is an explicit calculation of the Jacobian \eqref{jacobian} and that the multiplicity hypothesis of theorem \ref{pointwise} implies control of the multiplicity of the mapping \eqref{mapping}.  Computation of \eqref{jacobian} is easy: it equals $| \det \frac{\partial \Phi}{\partial x} (x, t^{(2)} - t^{(1)})|$ by linearity in $t$. Translating $t^{(2)}$ by $t^{(1)}$, the sublevel estimate \eqref{sublevel} clearly implies the sublevel hypothesis of lemma \ref{estimate}.  Regarding multiplicity: consider all tuples $(x,x',t^{(1)},t^{(2)})$ for which
\[ (x + t^{(1)},x'+ \Phi(x,t^{(1)}), x + t^{(2)}, x'+\Phi(x,t^{(2)})) \]
equals some specified constant tuple.  Subtracting first two blocks from the second two, we get that
$t^{(2)} - t^{(1)}$ and $\Phi(x,t^{(2)}-t^{(1)})$ are both constant on this set.  If the set contains any points where the Jacobian is nonzero, then the multiplicity hypothesis of theorem \ref{pointwise} will immediately bound the number of possible values that $x$ can attain on the set.  In the specific case of bilinearity, nonvanishing of the Jacobian implies global invertibility of the map as a function of $x$, hence the multiplicity will simply equal one.  Finally, if $x$ takes boundedly many values and $x + t^{(1)}, x+ t^{(2)}$ are constant, then $t^{(1)}$ and $t^{(2)}$ take only boundedly many values, so finally, $x'$ will take on at most boundedly many values as well.  Thus theorem \ref{pointwise} is a simple consequence of lemma \ref{estimate}.

Returning to the specific situation of theorem \ref{mainthm}, the Jacobian determinant will equal 
\[ J_Q(x,t^{(1)},t^{(2)}) =  | \det Q ( \cdot, t^{(2)} - t^{(1)}) |,\]
where by $Q (\cdot, t^{(2)} - t^{(1)})$ we mean the mapping $v \mapsto Q(v, t^{(2)} - t^{(1)})$.   When \eqref{mainclass} has nonvanishing rotational curvature for this $Q$, it means that
\[ \lambda \cdot Q(\cdot,t^{(2)}-t^{(1)}) \]
is a nonzero linear functional whenever $\lambda \neq 0$ and $t^{(2)} - t^{(1)} \neq 0$.  This means that $J_Q \neq 0$ when $t^{(2)} - t^{(1)} \neq 0$.
 We will see that this condition can be relaxed considerably without sacrificing anything except possibly the endpoint $L^\frac{3}{2} \rightarrow L^3$ estimate (in particular, without sacrificing any of the other estimates that one would obtain by interpolation with the trivial $L^1 \rightarrow L^1$ and $L^\infty \rightarrow L^\infty$ estimates).

The multiplicity requirement of lemma \ref{estimate} already having been established, it suffices to show that
\begin{equation}
\begin{split}
 \left|\set{ t \in \Omega}{ |\det Q(\cdot, t - t^{(i)})| \leq | \det Q(\cdot, t^{(2)} - t^{(1)})|, \ i=1,2}\right|& \\  \leq C_{\theta} |\det Q(\cdot,  t^{(2)} & - t^{(1)})|^{\theta}. 
 \end{split} \label{partial}
\end{equation}
Translating once again, it will suffice to show that
\begin{align}
 \left|\set{ t \in \Omega}{ |\det Q(\cdot, t )| \leq \epsilon } \right|   \leq C_{\theta} \epsilon^{\theta}. \label{full}
 \end{align}
 We note that the replacement of \eqref{partial} by \eqref{full} is where we lose the possibility of obtaining a restricted version of the endpoint estimate.  It is easy to construct examples for which \eqref{partial} holds in the case $\theta = 1$, but, for many of these examples, \eqref{full} will not be possible when $\theta = 1$.

The proof of theorem \ref{mainthm} follows from theorem \ref{pointwise} by demonstrating that almost every $T_{Q_u}$ (as defined in theorem \ref{mainthm}) satisfies the hypotheses of theorem \ref{pointwise}.  Because the multiplicity concerns have already been dispatched, the only issue is the sublevel set estimate.  We will establish the estimate almost everywhere by means of an elementary rearrangement argument and Tchebyshev's inequality.  First the rearrangement argument:

\begin{lemma}
Suppose $F$ is nonnegative and decreasing on $[0,\infty)$, and let $\varphi : \R^n \rightarrow \R$ be smooth and have the property that $\frac{\partial^2}{\partial u_j^2} \varphi(u) = 0$ for all $j=1,\ldots,n$ and all $u \in \R^n$.  Then
\[ \int_{B} F(|\varphi(u)|) du \leq \int_{B} F \left( \left| u_1 \cdots u_n  \frac{\partial^n \varphi}{\partial u_1 \cdots \partial u_n}(0) \right| \right) du \] 
when $B$ is any $n$-fold \label{rearrange} product of intervals centered at the origin.
\end{lemma}
\begin{proof}
In one dimension, the key inequality is that
\[ \int_{-c}^c F(|a u + b|) du \leq \int_{-c}^c F(|au|) du \]
which is a simple consequence of monotonicity: if $a = 0$ the appeal to monotonicity is direct.  Otherwise, we have
\[  \int_{-c}^c F(|a u + b|) du = \int_{-c + a^{-1} b}^{c + a^{-1} b} F(|au|) du. \]
If $b \neq 0$, then some positive fraction of the interval $[-c + a^{-1} b, c + a^{-1} b]$ lies outside $[-c,c]$.  However, the supremum of $F(|au|)$ outside $[-c,c]$ is less than the infimum of $F(|au|)$ inside $[-c,c]$ by monotonicity.  Thus the integral can only be made larger by shifting the support to be centered at the origin.  For $n > 1$, the lemma follows by Fubini and induction, applying this one-dimensional argument iteratively in each coordinate direction.
\end{proof}
To finish the proof of theorem \ref{mainthm}, we choose
\[ F(s) := \int_0^1 |s|^{-1 + \rho} \rho^{2n} d \rho. \]
Clearly $F$ is positive and decreasing on $[0,\infty)$.   Fix any $Q : \R^n \times \R^n \rightarrow \R^n$ bilinear and let $Q_u$ be the bilinear map
\[ Q_u(x,t) := Q(x,t) + \sum_{j=1}^n u_j x_j t_j \]
as defined in theorem \ref{mainthm}.  We will apply lemma \ref{rearrange} to $\varphi(u) := \det Q_u(\cdot , t)$.
It is easy to see that the second derivatives $\partial_{u_j}^2 \varphi(u)$ all vanish identically.  Consequently
\begin{align*}
 \int_{ [-1,1]^n} & F( |\det Q_u (\cdot, t)|) d u  \leq \int_{[-1,1]^n} F(| u_1 \cdots u_n t_1 \cdots t_n|) d u. \end{align*}
 If we integrate both sides of this inequality over $[-1,1]^n$ in the $t$ variables, then by Fubini we have
 \begin{align*}
  \int_{[-1,1]^n} & \left[ \int_{[-1,1]^n} F( |\det Q_u (\cdot,  t)|) dt \right] d u \\
& \leq \int_0^1 \rho^{2n} \left( \int_{-1}^1 |s|^{-1 + \rho} d s \right)^{2n} d \rho = 2^{2n}  < \infty. \end{align*}
By homogeneity (and rescaling as necessary), we conclude that
\begin{equation} \int_{[-1,1]^n} F( |\det Q_u (\cdot, t)|) dt < \infty  \label{finiteness}
\end{equation}
for almost every $u$ and any fixed $Q$.
We can bound $F$ from below as follows (splitting into cases $|s| \leq 1$, $|s| \geq 1$ and $\rho \leq 1- \theta$, $\rho \geq 1-\theta$):
\[ \frac{\min \{ (1-\theta)^{2n+1} , 1 - (1-\theta)^{2n+1} \}}{2n+1} |s|^{- \theta} \leq \int_0^1 \rho^{2n} |s|^{-1+\rho} d \rho = F(s) \]
for any $\theta \in (0,1)$ and any $s \in \R$. By Tchebyshev's inequality, for almost every choice of $u$ as described above we have that for every $\theta \in (0,1)$, there will be a constant $C_{u,\theta}$ such that \[ \left| \set{t \in [-1,1]^n}{ |\det Q_u (\cdot, t)| \leq \epsilon} \right|  \leq C_{u,\theta} \epsilon^{\theta} \]
uniformly for all $\epsilon > 0$.  By theorem \ref{pointwise}, the proof of theorem \ref{mainthm} is complete.

\bibliography{mybib}

\def\cprime{$'$}
\begin{bibdiv}
\begin{biblist}

\bib{bg2011}{article}{
      author={Bourgain, Jean},
      author={Guth, Lawrence},
       title={Bounds on oscillatory integral operators},
        date={2011},
     journal={C. R. Math. Acad. Sci. Paris},
      volume={349},
      number={3-4},
       pages={137\ndash 141},
}

\bib{christthesis}{thesis}{
      author={Christ, Michael},
       title={Restriction of the fourier transform to submanifolds of low
  codimension},
        type={Ph.D. Thesis},
 institution={University of Chicago},
        date={1982},
}

\bib{christ1998}{article}{
      author={Christ, Michael},
       title={Convolution, curvature, and combinatorics: a case study},
        date={1998},
     journal={Internat. Math. Res. Notices},
      number={19},
       pages={1033\ndash 1048},
}

\bib{dlw2009}{article}{
      author={Dendrinos, Spyridon},
      author={Laghi, Norberto},
      author={Wright, James},
       title={Universal {$L^p$} improving for averages along polynomial curves
  in low dimensions},
        date={2009},
     journal={J. Funct. Anal.},
      volume={257},
      number={5},
       pages={1355\ndash 1378},
}

\bib{eo2010}{article}{
      author={Erdo{\u{g}}an, M.~Burak},
      author={Oberlin, Richard},
       title={Estimates for the {$X$}-ray transform restricted to 2-manifolds},
        date={2010},
     journal={Rev. Mat. Iberoam.},
      volume={26},
}

\bib{frs1984}{article}{
      author={Friedland, S.},
      author={Robbin, J.~W.},
      author={Sylvester, J.~H.},
       title={On the crossing rule},
        date={1984},
     journal={Comm. Pure Appl. Math.},
      volume={37},
      number={1},
       pages={19\ndash 37},
}

\bib{gs2002}{inproceedings}{
      author={Greenleaf, Allan},
      author={Seeger, Andreas},
       title={Oscillatory and {F}ourier integral operators with degenerate
  canonical relations},
        date={2002},
   booktitle={Proceedings of the 6th {I}nternational {C}onference on {H}armonic
  {A}nalysis and {P}artial {D}ifferential {E}quations ({E}l {E}scorial, 2000)},
       pages={93\ndash 141},
}

\bib{gressman2013}{article}{
      author={Gressman, Philip~T.},
       title={Uniform {S}ublevel {R}adon-like {I}nequalities},
        date={2013},
     journal={J. Geom. Anal.},
      volume={23},
      number={2},
       pages={611\ndash 652},
}

\bib{gs1977}{book}{
      author={Guillemin, Victor},
      author={Sternberg, Shlomo},
       title={Geometric asymptotics},
   publisher={American Mathematical Society},
     address={Providence, R.I.},
        date={1977},
        note={Mathematical Surveys, No. 14},
}

\bib{oberlin2000II}{article}{
      author={Oberlin, Daniel~M.},
       title={Convolution with measures on hypersurfaces},
        date={2000},
     journal={Math. Proc. Cambridge Philos. Soc.},
      volume={129},
      number={3},
       pages={517\ndash 526},
}

\bib{ps1986I}{article}{
      author={Phong, D.~H.},
      author={Stein, E.~M.},
       title={Hilbert integrals, singular integrals, and {R}adon transforms.
  {I}},
        date={1986},
     journal={Acta Math.},
      volume={157},
      number={1-2},
       pages={99\ndash 157},
}

\bib{ps1986II}{article}{
      author={Phong, D.~H.},
      author={Stein, E.~M.},
       title={Hilbert integrals, singular integrals, and {R}adon transforms.
  {II}},
        date={1986},
     journal={Invent. Math.},
      volume={86},
      number={1},
       pages={75\ndash 113},
}

\bib{steinha}{book}{
      author={Stein, Elias~M.},
       title={Harmonic analysis: real-variable methods, orthogonality, and
  oscillatory integrals},
      series={Princeton Mathematical Series},
   publisher={Princeton University Press},
     address={Princeton, NJ},
        date={1993},
      volume={43},
        note={With the assistance of Timothy S. Murphy, Monographs in Harmonic
  Analysis, III},
}

\bib{stovall2011}{article}{
      author={Stovall, Betsy},
       title={{$L^p$} improving multilinear {R}adon-like transforms},
        date={2011},
     journal={Rev. Mat. Iberoam.},
      volume={27},
      number={3},
       pages={1059\ndash 1085},
}

\bib{tw2003}{article}{
      author={Tao, Terence},
      author={Wright, James},
       title={{${L}\sp p$} improving bounds for averages along curves},
        date={2003},
     journal={J. Amer. Math. Soc.},
      volume={16},
      number={3},
       pages={605\ndash 638},
}

\end{biblist}
\end{bibdiv}

\end{document}